\documentclass[11pt,reqno]{amsart}

\usepackage{amsxtra,amssymb,amsthm,amsmath,amscd,url,epsfig}
\usepackage[latin1]{inputenc}
\usepackage{dsfont}
\usepackage{hyperref}
\usepackage{fullpage}

\usepackage{amssymb}
\usepackage{comment}
\usepackage{latexsym}
\usepackage{amsbsy}
\usepackage{amsfonts}
\usepackage{hyperref}
\usepackage{graphicx}
\usepackage{enumerate}
\usepackage{color}

\usepackage{tikz}

\def\marginpar#1{\ignorespaces}

\newtheorem{theorem}{Theorem}[section]
\newtheorem{lemma}[theorem]{Lemma}
\newtheorem{proposition}[theorem]{Proposition}

\newtheorem{definition}[theorem]{Definition}

\newtheorem{example}[theorem]{Example}
\newtheorem{fact}[theorem]{Fact}
\newtheorem{remark}[theorem]{Remark}
\numberwithin{equation}{section}

\newcommand{\eps}{\varepsilon}

\renewcommand{\Re}{\operatorname{Re}}
\renewcommand{\Im}{\operatorname{Im}}

\renewcommand{\leq}{\leqslant}
\renewcommand{\geq}{\geqslant}

\newcommand{\Zz}{\mathbf{Z}}

\newcommand{\Rr}{\mathbf{R}}

\newcommand{\expect}{\text{\boldmath$E$}}

\newcommand{\ra}{\rightarrow}

\numberwithin{equation}{section}

\begin{document}

\title{A characterization of limiting functions arising in Mod-* convergence} % \thanks is optional. Insert line breaks with \\

\authors{%
 Emmanuel Kowalski\footnote{ETH Zurich, D-MATH, R\"amistrasse 101,
  8092 Z\"URICH, Switzerland.
    \email{kowalski@math.ethz.ch}}
  \and 
  Joseph Najnudel\footnote{Institut de Math\'ematiques de Toulouse, Universit\'e Paul Sabatier, 118 route de Narbonne, 
31062 TOULOUSE, France.
    \email{joseph.najnudel@math.univ-toulouse.fr}}
    \and
  Ashkan Nikeghbali\footnote{Institut f\"ur Mathematik, Universit\"at Z\"urich, Winterthurerstrasse 190,
8057-Z\"urich, Switzerland.
    \email{ashkan.nikeghbali@math.uzh.ch}}
    }

  \date{}
 % Separate items with ;

\begin{abstract} In this note, we characterize the limiting functions in mod-Gausssian convergence; our approach sheds a new light on the nature of mod-Gaussian convergence as well. Our results in fact more generally apply to  mod-* convergence, where * stands for any family of probability distributions whose Fourier transforms do not vanish. We moreover provide new examples, including two new examples of (restricted) mod-Cauchy convergence
from arithmetics related to Dedekind sums and the linking number of modular geodesics.
\end{abstract}

%%%%%%%%%%%%%%%%%%%%%%%%%%%%%%%%%%%%%%%%%%%%%%%%%%%%%%%%%%%%%%%%%%%
%%                                                               %%
%% Please add your own macros and environments below:            %%
%%                                                               %%
%% If possible, avoid using \def and use instead \newcommand     %%
%% If possible, avoid defining your own environments, and use    %%
%% instead the environments already defined by ejpecp:           %%
%%  assumption, assumptions, claim, condition, conjecture,       %%
%%  corollary, definition, definitions, example, exercise, fact, %%
%%  facts, heuristics, hypothesis, hypotheses, lemma, notation,  %%
%%  notations, problem, proposition, remark, theorem             %%
%%                                                               %%
%%%%%%%%%%%%%%%%%%%%%%%%%%%%%%%%%%%%%%%%%%%%%%%%%%%%%%%%%%%%%%%%%%%

\maketitle
\section{Introduction}
In \cite{JKN} a new type of convergence  which can be viewed as a refinement of the central limit theorem was proposed, following the idea that, given a sequence of random variables, one looks  for the convergence of the renormalized  sequence of characteristic functions  rather than the convergence of the renormalized sequence of the given random variables. More precisely the following definitions were introduced:
\begin{definition}[\cite{JKN}]Let $(\Omega,\mathcal F,\mathbb P)$ be a probability space and let $(X_n)_{n\geq0}$ be a sequence of random variables defined on this probability space.
\begin{enumerate}
\item We say that $(X_n)_{n\geq0}$ converges in the mod-Gaussian sense with parameters $(m_n,\sigma_n^2)$ and limiting function $\Phi(\lambda)$ if the following convergence holds locally uniformly for $\lambda$:
\begin{equation}\label{defmg}
\lim_{N\to\infty} \exp\left(-i m_N \lambda+\frac{\sigma_N^2\lambda^2}{2}\right)\mathbb E\left[\exp\left(i\lambda X_N\right)\right]=\Phi(\lambda),
\end{equation}(we have normalized by the characteristic function of  Gaussian random variables with mean $m_n$ and variance $\sigma_n^2$).
\item We say that the sequence $(X_n)_{n\geq0}$ converges in the mod-Poisson sense with parameter $\gamma_N$ and limiting function $\Phi$ if the following convergence holds locally uniformly for $\lambda$:
\begin{equation}\label{defmp}
\lim_{N\to\infty} \exp\left( -\gamma_N\left(e^{i\lambda}-1\right)\right)\mathbb E\left[\exp\left(i\lambda X_N\right)\right]=\Phi(\lambda),
\end{equation}(we have normalized by the characteristic function of  Poisson random variables with mean $\gamma_n$).

\end{enumerate}
\end{definition}
In fact, as pointed out in \cite{JKN}, one can more generally study the convergence of the characteristic functions after renormalization with any family of characteristic functions which do not vanish: with this more general situation in mind, we talk about \textit{mod-*} convergence. In a series of works \cite{JKN, KN, KN2, BKN,DKN} the authors establish that mod-* convergence occurs in many situations in number theory, random matrix theory, probability theory, random permutations and combinatorics and prove that under some extra assumptions, mod-* convergence may imply results such as local limit theorems, distributional approximations or precise large deviations. It should be noted that mod-* convergence usually implies convergence in law of the random variables $X_N$, possibly after  rescaling, which corresponds to most interesting studied cases where $m_n=0,\sigma_N\to\infty$ in (\ref{defmg}) or $\gamma_n\to\infty$ in (\ref{defmp}). Moreover it is shown in \cite{JKN} and \cite{KN2} that the limiting 
function sheds some new light into the connections between number theoretic objects and their naive probabilistic models. Roughly speaking, naive probabilistic models are based on the wrong assumptions that primes behave independently of each other but yet they can predict central limit theorems, such as Selberg's central limit theorem for the Riemann zeta function or the Erdos-Kac central limit theorem for the total number of distinct prime divisors of integers. However at the level of mod-Gaussian or mod-Poisson convergence, they fail to predict the correct behavior and a correction factor  appears in the limiting function to account for the lack of independence. Hence the limiting function seems to carry some information about the dependence among prime numbers. It thus seems natural to ask what  the possible limiting functions can be in the framework of mod-* convergence and this question was left open in \cite{JKN}. 

In this paper, we propose a characterization of the limiting functions. Let $S_0$ be the set of functions which can be obtained as the characteristic function of a real random variable, divided by 
the characteristic function of a gaussian random variable. It is clear that $S_0$ is contained in the set $S$ of the continuous functions $\phi$ from $\mathbb{R}$ 
to $\mathbb{C}$ such that $\phi(0) = 1$ and  $\phi(-\lambda) = \overline{\phi(\lambda)}$ for all $\lambda \in \mathbb{R}$. The converse is not true: it 
is clear that if a function $\phi$ in $S$ tends to infinity faster than $\lambda \mapsto e^{\sigma^2 \lambda^2/2}$ when $|\lambda|$ goes to infinity, for all 
$\sigma > 0$, then $\phi \notin S_0$. However,  the following result holds:

\begin{theorem} \label{p1}
The set $S_0$ is dense in $S$ for the topology of the uniform convergence on compact sets. 
\end{theorem}
The next section is devoted to a complete and short proof of this result and on another possible proof based on the study of mod-Gaussian convergence for sums of i.i.d. random variables. We also propose the larger framework where mod-* convergence only holds on a finite interval. We moreover provide two new examples of mod-Cauchy convergence  from arithmetics related to Dedekind sums and the linking number of modular geodesics, thus strengthening the relevance of this framework in number theory as well.

\section{Proofs of Theorem \ref{p1}}
\subsection{Analytic proof}

Let $P$ be a polynomial with real coefficients, such that $P(0) = 1$. For all $\sigma > 0$, let 
us define the function $f_{\sigma}$ from $\mathbb{R}$ to $\mathbb{R}$ by 
$$f_{\sigma} (x) = \frac{1}{\sigma \sqrt{2 \pi}} \, e^{-x^2/2\sigma^2},$$ 
and the function $g_{P,\sigma}$ from $\mathbb{R}$ to $\mathbb{R}$ by
$$g_{P,\sigma}(x) = \frac{\sigma}{\sigma +1} \left( P(D)(f_{\sigma}) (x) + \frac{1}{2\sigma^2 \sqrt{2 \pi}}  e^{-x^2/8 \sigma^2} \right),$$
where $D$ denotes the operator of differentiation of functions (e.g., for $P(x) = x^2 + 1$, $P(D) (f_{\sigma}) = f''_{\sigma} + f_{\sigma})$. 
We first establish a lemma:
\begin{lemma}
For any real polynomial $P$ with constant term $1$, there exists $\sigma_0 > 0$ such that for all $\sigma \geq \sigma_0$, 
 $g_{P,\sigma}(x)$ is a nonnegative function. 
\end{lemma}
\begin{proof}
Without loss of generality, we can assume that $\operatorname{deg}(P) \geq 1$, i.e. $P \neq 1$ (for $P=1$ the result is trivial). Now
 $f_{\sigma} (x) = \frac{f_1(x/\sigma)}{\sigma}$ and then, by taking the $k$-th derivative, 
$$f_{\sigma}^{(k)} (x) = \frac{f_1^{(k)} (x/\sigma)}{\sigma^{k+1}}$$
for all $\sigma> 0$, $x \in \mathbb{R}$, $k \geq 0$. From the expression of the derivatives of $f_1$ in terms of Hermite polynomials, one 
deduces that there exists a constant $C_P> 1$, depending only on the polynomial $P$,  such that 
\begin{equation}
|P(D) (f_{\sigma}) (x) - f_{\sigma}(x)| \leq C_P \, \left( \frac{1}{\sigma} +  \frac{|x|}{\sigma^2} + \left( \frac{|x|}{\sigma^2} \right)^{\operatorname{deg} (P)} \right)  f_{\sigma} (x)
\label{ndn}
\end{equation}
for all $\sigma> 1$, $x \in \mathbb{R}$ (recall that $P-1$ has no constant term). Let us first suppose that $\sigma >1$ and $|x| \leq \sigma^{3/2}$. 
In this case, $|x|/\sigma^2 \leq 1/\sqrt {\sigma}$, and then, from \eqref{ndn}:
$$|P(D) (f_{\sigma}) (x) - f_{\sigma}(x)|  \leq \frac{3 C_P}{\sqrt{\sigma}} \, f_{\sigma}(x),$$
which implies that $P(D) (f_{\sigma}) (x) \geq 0$, and a fortiori $g_{P,\sigma}(x) \geq 0$, for $\sigma \geq 9 C_P^2$. 
Let us now suppose that $|x| > \sigma^{3/2}$. In this case, for $\sigma > 3$, 
\begin{align*}
|P(D) (f_{\sigma}) (x)| &  \leq C_P \, \left( 1+ \frac{1}{\sigma} +  \frac{|x|}{\sigma^2} + \left( \frac{|x|}{\sigma^2} \right)^{\operatorname{deg} (P)} \right)
 f_{\sigma} (x) \leq  3 C_P \, \left( 1 +  \left( \frac{|x|}{\sigma^2} \right)^{\operatorname{deg} (P)}\right)
 f_{\sigma} (x) \\ & \leq 3 C_P ( \operatorname{deg} (P))! \, e^{|x|/\sigma^2}  f_{\sigma} (x)
 \leq \frac{3 C_P ( \operatorname{deg} (P))!}{\sigma \sqrt{2 \pi}} e^{|x|/\sigma^2 - x^2/2\sigma^2} \\ & 
 \leq  C_P ( \operatorname{deg} (P))! e^{-x^2/4\sigma^2}
 \end{align*}
 \noindent
 the third inequality coming from the Taylor expansion of the exponential function, and the last inequality coming from the fact that 
 $\sigma>3$, and then $|x| > \sigma^{3/2} > 3^{3/2} > 4$, which implies that $e^{|x|/\sigma^2} \leq e^{x^2/4 \sigma^2}$. 
 One deduces: 
$$ P(D)(f_{\sigma}) (x) + \frac{1}{2\sigma^2 \sqrt{2 \pi}}  e^{-x^2/8 \sigma^2}  \geq 
e^{-x^2/4 \sigma^2} \left( \frac{1}{2\sigma^2  \sqrt{2 \pi}}  e^{x^2/8 \sigma^2}  -  C_P ( \operatorname{deg} (P))!  \right),$$
which implies that $g_{P,\sigma}(x) \geq 0$, provided that 
\begin{equation}
\frac{1}{2\sigma^2 \sqrt{2 \pi}}  e^{x^2/8 \sigma^2}  \geq C_P ( \operatorname{deg} (P))! \label{bouh}
\end{equation}
Now, since
$$ \frac{1}{2\sigma^2 \sqrt{2 \pi}}  e^{x^2/8 \sigma^2}  \geq  \frac{1}{2\sigma^2 \sqrt{2 \pi}}  e^{\sigma/8},$$
the inequality \eqref{bouh} holds for all $\sigma$ large enough, depending only on $P$. 
\end{proof}
\noindent
Once the positivity of $g_{P,\sigma}$ is proven (for $\sigma$ large enough, depending only on $P$), let us compute its Fourier transform: one checks that for all 
$\lambda \in \mathbb{R}$, 
$$\int_{-\infty}^{\infty} g_{P,\sigma} (x) e^{i \lambda x} dx = \frac{\sigma}{\sigma +1} \left( P(-i \lambda) e^{-\sigma^2 \lambda^2/2} + \frac{1}{\sigma} e^{-2 \sigma^2 \lambda^2} \right).$$
In particular, the value of the Fourier transform at $\lambda =0$ is equal to one, which implies that $g_{P, \sigma}$ is in fact a probability density. 
Hence, the following function is in $S_0$:
$$\lambda \mapsto \frac{\sigma}{\sigma +1} \left( P(-i \lambda) + \frac{1}{\sigma} e^{-3 \sigma^2 / 2 \lambda^2} \right).$$
By letting $\sigma \rightarrow \infty$, one deduces that the adherence of $S_0$, for the topology of uniform convergence on compact sets, contains the function
$$\lambda  \mapsto P(-i \lambda),$$
and then all the functions in $S$, by the Stone-Weierstrass theorem.

\subsection{Probabilistic proof: mod-Gaussian convergence for sums of i.i.d. random variables}
It is natural to ask whether there exists a general result of mod-Gaussian convergence for sums of i.i.d. random variables like there exists a central limit theorem. The answer is positive and provides in fact an alternative proof to  Theorem \ref{p1}. The result also outlines the interesting fact that mod-Gaussian convergence is closely related to cumulants. More precisely, we have the following result:
\begin{proposition} \label{cum1}
Let $k \geq 2$ be an integer, and let $(X_n)_{n \geq 1}$ be a sequence of i.i.d. variables in $L^{r}$ for some $r > k+1$, such that the $k$ first moments of $X_1$ are
 the same as the corresponding moments of a
standard gaussian variable. Then, the sequence of variables 
$$\left( \frac{1}{N^{1/(k+1)}} \, \sum_{n = 1}^N  X_n \right)_{N \geq 1}$$
converges in the mod-gaussian sense, with the sequence of means and variances 
$$m_N = 0, \: \sigma^2_N = N^{(k-1)/(k+1)},$$
to the function 
$$\lambda \mapsto e^{ (i \lambda)^{k+1} c_{k+1}/(k+1)!},$$
where $c_{k+1}$ denotes the $(k+1)$-th cumulant of $X_1$. 
\end{proposition}
\begin{remark}
 Intuitively, this mod-gaussian convergence suggests to approximate the distribution of the renormalized
 partial sums of 
$(X_n)_{n \geq 1}$ by the convolution of a gaussian density and a function $H_{k, c_{k+1}}$ whose 
Fourier transform is $\lambda \mapsto e^{ (i \lambda)^{k+1} c_{k+1}/(k+1)!}$. The function $H_{k, c_{k+1}}$
 is not 
a probability density, since it takes some negative values: it appears in a paper by 
Diaconis and Saloff-Coste \cite{DSC} on convolutions of measures on $\mathbb{Z}$. 
\end{remark}

\begin{proof}
On can assume $r \in (k+1,k+2)$, and then one has for all $\lambda \in \mathbb{R}$,
$$\left| e^{i\lambda} - \sum_{j = 0}^{k+1} \frac{(i\lambda)^j}{j!} \right| \leq |\lambda|^r.$$
Hence, if $\phi$ denotes that characteristic function of $X_1$, and $(\mu_j)_{0 \leq j \leq k+1}$ the first successive moments of $X_1$, one has
$$\phi(\lambda) = \sum_{j = 0}^{k+1} \mu_j \, \frac{(i\lambda)^j}{j!} + O(|\lambda|^r),$$
when $\lambda$ goes to zero. 
Now, $(\mu_j)_{0 \leq j \leq k}$ and $\mu_{k+1} - c_{k+1}$ are also the first moments of the standard Gaussian variable, hence,
$$e^{-\lambda^2/2} = \sum_{j = 0}^{k+1} \mu_j \, \frac{(i\lambda)^j}{j!}  - c_{k+1} \, \frac{(i\lambda)^{k+1}}{(k+1)!} + O(|\lambda|^{k+2}).$$
Therefore,
$$\phi(\lambda) = e^{-\lambda^2/2} + c_{k+1} \, \frac{(i\lambda)^{k+1}}{(k+1)!} + O(|\lambda|^r),$$
and then
$$\phi(\lambda) \, e^{\lambda^2/2}  = 1 + c_{k+1} \, \frac{(i\lambda)^{k+1}}{(k+1)!} + O(|\lambda|^r).$$
One deduces that for fixed $\lambda$, 
\begin{equation}
[\phi(\lambda/N^{1/(k+1)})]^N e^{\lambda^2N^{(k-1)/(k+1)}/2} \underset{N \rightarrow \infty}{\longrightarrow} e^{c_{k+1} (i\lambda)^{k+1}/(k+1)!} \label{tety}
\end{equation}
Now, the left-hand side of \eqref{tety} is the characteristic function of 
$$ \frac{1}{N^{1/(k+1)}} \, \sum_{n = 1}^N  X_n,$$
divided by the characteristic function of a centered gaussian variable of variance $N^{(k-1)/(k+1)}$. 
\end{proof}
\begin{example}
If $(X_n)_{n \geq 1}$ are i.i.d. variables, such that $\mathbb{P}[X_1 =1] = \mathbb{P}[X_1 = -1] = 1/2$, then 
$$\left(\frac{1}{N^{1/4}} \, \sum_{n = 1}^N  X_n \right)_{N \geq 1}$$
converges, in the mod-gaussian sense, with the sequence of means and variances
$$m_N = 0, \; \sigma^2_N = \sqrt{N},$$
to the function $\lambda \mapsto e^{-\lambda^4/12}$. 
\end{example}
\noindent
{\bf Proof of Theorem 1.2.} After a possible multiplication of the variables $(X_n)_{n \geq 1}$ by a constant, one can obtain, from Proposition \ref{cum1},
all the exponential of monomials in $i \lambda$ as mod-gaussian limits, since the cumulant $c_{k+1}$ can be positive or negative. 
By taking sums of independent random variables, one deduces the exponential of all the polynomials in $i \lambda$ without constant term.  Now, by Stone-Weierstrass theorem, 
one obtains the exponential of all the continuous functions $f$ such that $f(- \lambda) = \overline{f(\lambda)}$ and $f(0) = 0$, i.e. all the non-vanishing 
functions in $S$. By taking approximations which avoid the zeros (which are finitely many), one obtains all the polynomial functions in $S$,  and by using 
again the Stone-Weierstrass theorem, one deduces another proof of Theorem \ref{p1}. 
 \begin{remark}
One may in fact go further in the cumulants approach to mod-Gaussian convergence for an arbitrary sequence of random variables $(X_n)$. In this case, the cumulants depend on $n$ and one needs to control the growth of the cumulants of order $k$ higher than $2$ as functions  of $(k,n)$. This approach is useful in some combinatorial framework and under some analytic assumptions one deduces precise large deviations estimates (with a good control on the error terms) from mod-* convergence. This is the topic of a forthcoming work.
\end{remark}

\section{Further examples and remarks}
All the limiting functions obtained from mod-* convergence correspond to functions which are in the space $S$. In other words, they
 can always be obtained as 
mod-gaussian limits. The functions of $S$ can also be viewed as the Fourier transforms of some special kind of distributions. 
 Indeed, let $\mathcal{E}$ be the space of functions from $\mathbb{R}$ to $\mathbb{R}$, generated by the functions 
$x \mapsto \cos(\mu x)$ for $\mu \geq 0$ and $x \mapsto \sin(\mu x)$ for $\mu > 0$. These functions form a basis of $\mathcal{E}$. 
Indeed, if for $p \geq 0$, $q \geq 0$, 
$\mu_1 > \mu_2 > \dots > \mu_p \geq 0$, $\mu'_1 > \dots > \mu'_q > 0$, $\alpha_1, \dots, \alpha_p, \alpha'_1, \dots \alpha'_q \neq 0$, 
the function 
$$g : x \mapsto \sum_{j = 1}^p \alpha_j \cos(\mu_j x) +  \sum_{k = 1}^q \alpha'_k \sin(\mu'_k x)$$
vanishes for all $x \in \mathbb{R}$, it vanishes for all $x \in \mathbb{C}$, since it is an entiere function. 
If $p \geq 1$, then for $y$ real and tending to infinity, 
$$\Re(g(iy)) = \alpha_1 e^{\mu_1 y} \left(\frac{1}{2}  \mathds{1}_{\mu_1 > 0} + \mathds{1}_{\mu_1 = 0} \right) + o(e^{\mu_1 y}),$$
and if $q \geq 1$, 
$$\Im(g(iy)) = - \frac{\alpha'_1}{2} e^{\mu'_1 y} + o(e^{\mu'_1 y}),$$
which contradicts the fact that $g$ is identically zero.  

One can then define the distributions with space of test functions $\mathcal{E}$ as the linear forms on this space. 
The following result clearly holds:
\begin{lemma}
A distribution $\mathcal{D}$, defined as a linear form on $\mathcal{E}$, is   
characterized 
by its values $\psi_{\mathcal{D}}(\mu)$ at the functions $x \mapsto \cos(\mu x)$ ($\mu \geq 0$),
and $\psi'_{\mathcal{D}}(\mu)$ at the functions $x \mapsto \sin(\mu x)$ ($\mu > 0$).
Moreover, if the distributions are canonically extended to complex test functions, then 
for $\lambda \in \mathbb{R}$, the image of $x \mapsto e^{i \lambda x}$ by $\mathcal{D}$ 
is given by 
$$\phi_{\mathcal{D}}(\lambda) = \psi_{\mathcal{D}}(\lambda) + i \psi'_{\mathcal{D}}(\lambda)$$
for $\lambda > 0$, 
$$\phi_{\mathcal{D}}(\lambda) = \psi_{\mathcal{D}}(0) $$
for $\lambda = 0$ and
$$\phi_{\mathcal{D}}(\lambda) = \psi_{\mathcal{D}}(|\lambda|) - i \psi'_{\mathcal{D}}(|\lambda|)$$
for $\lambda < 0$.
\end{lemma}
The function $\phi_{\mathcal{D}}$ can be viewed as the Fourier transform of $\mathcal{D}$.
The following also holds: 
\begin{lemma}
The equation $\phi_{\mathcal{D}} (-\lambda) 
= \overline{\phi_{\mathcal{D}} (\lambda)}$ is satisfied 
for all $\lambda \in \mathbb{R}$. Moreover, the map $(\psi_{\mathcal{D}}, \psi'_{\mathcal{D}})
\mapsto \phi_{\mathcal{D}}$ from $\mathcal{F}_1 \times \mathcal{F}_2$ to $\mathcal{G}$ is bijective, where
$\mathcal{F}_1$ is the 
space of functions from $\mathbb{R}_+$ to $\mathbb{R}$, $\mathcal{F}_2$ is the 
space of functions from $\mathbb{R}^*_+$ to $\mathbb{R}$, and $\mathcal{G}$ is the space of functions 
$\phi$ from $\mathbb{R}$ to $\mathbb{C}$ 
satisfying the equation $\phi (-\lambda) = \overline{\phi (\lambda)}$. The space of distributions 
on $\mathcal{E}$
is in bijection with $\mathcal{G}$, via the Fourier transform $\mathcal{D} \mapsto \phi_{\mathcal{D}}$.  
\end{lemma}
Since $S$ is included in $\mathcal{G}$, the functions in $S$ can be viewed, via inverse Fourier transform, as distributions with 
test space $\mathcal{E}$. Note that these distributions can be very singular, since we have a priori no control on the 
behavior of their Fourier transform at infinity: in general, they cannot be identified with tempered distributions in the usual sense. 
If $\mathcal{D}_1$ and $\mathcal{D}_2$ are two distributions with test space $\mathcal{E}$, one can define 
their convolution $\mathcal{D}_1 * \mathcal{D}_2$ as the distribution whose Fourier transform is the product $\phi_{\mathcal{D}_1} \phi_{\mathcal{D}_2}$.
If $\phi_{\mathcal{D}_2}$ nowhere vanishes, then the {\it deconvolution} of $\mathcal{D}_1$ by $\mathcal{D}_2$ is the unique distribution $\mathcal{D}$ such that 
$\mathcal{D} * \mathcal{D}_2 = \mathcal{D}_1$: one has $\phi_{\mathcal{D}} = \phi_{\mathcal{D}_1} / \phi_{\mathcal{D}_2}$. 
 
Now, the mod-gaussian convergence can be interpreted as follows: if the sequence of distributions
$(\mathcal{L}_n)_{n \geq 1}$ converges in the mod-gaussian sense, with the sequence of parameters $(m_n)_{n \geq 1}$ and $(\sigma^2_n)_{n \geq 1}$, to a 
function $\phi \in S$, then the deconvolution of $\mathcal{L}_n$ by the gaussian distribution $\mathcal{N} (m_n, \sigma^2_n)$ converges, in the 
sense of the distributions with test space $\mathcal{E}$, to the 
inverse Fourier transform of $\phi$ when $n$ goes
to infinity.

On the other hand, it is possible to enlarge the space of possible limit functions of mod-* convergence, by considering a weaker convergence.

\begin{definition}
Let $a > 0$, let $(\mathcal{L}_n)_{n \geq 1}$ be a sequence of probability distributions, and let $(\mathcal{M}_n)_{n \geq 1}$ be a sequence of 
probability distributions whose Fourier transforms do not vanish on the interval $(-a,a)$. Then 
$(\mathcal{L}_n)_{n \geq 1}$ converges $a$-mod-$(\mathcal{M}_n)_{n \geq 1}$ to a function $\phi$ from $(-a,a)$ to $\mathbb{C}$, 
if and only if the quotient of the Fourier transform of $\mathcal{L}_n$ by the Fourier transform of $\mathcal{M}_n$ converges to $\phi$, uniformly 
on all compact sets included in $(-a,a)$.
\end{definition}
\noindent 
The set of possible limiting functions is  given in the next proposition.
\begin{proposition}
The set of all the possible limits of $a$-mod-* convergence is the space $S_a$ of continuous functions $\phi$ from $(-a,a)$ to $\mathbb{C}$, such that 
$\phi(0) =1$ and $\phi(-\lambda) = \overline{\phi(\lambda)}$ for all $\lambda \in (-a,a)$. Moreover, all the functions in $S_a$ can be obtained from
$a$-mod-gaussian limit.
\end{proposition}
\begin{proof}
It is clear that all the $a$-mod-* limits are in $S_a$. Conversely, from Theorem \ref{p1}, all the restrictions to $(-a,a)$ of functions in $S$ can be obtained 
as $a$-mod-gaussian limits. The set of functions obtained in this way is dense in $S_a$, for the uniform convergence on compact subsets of $(-a,a)$. 
\end{proof}
\begin{remark}
The convergence described here is quite weak. In particular, for $\mathcal{M}_n = \delta_0$ (Dirac measure at zero), it does not implies convergence in law. Indeed, 
the Fourier transform of a probability distribution does not characterize it after restriction to a finite interval. For example 
the measure $(1- \cos x)/(\pi x^2) dx$
has Fourier transform
$$\lambda \mapsto (1-|\lambda|)_+,$$
the measure $\delta_0/2 + (1- \cos (x/2))/(\pi x^2) dx$ has Fourier transform
$$\lambda \mapsto 1/2 \, + (1/2-|\lambda|)_+,$$
and these two Fourier transforms coincide on the interval $[-1/2,1/2]$.
\end{remark}
\noindent
For a concrete example of $a$-mod-Gaussian convergence, we refer to Example 4 from \cite{KN} which is taken from random matrix theory and which is essentially due
to Wieand \cite{wieand}. Let $T_N\in U(N)$ be a random unitary matrix which is Haar distributed.  All eigenvalues are then on the unit circle. We consider the discrete valued random variable counting the number of eigenvalues lying in some fixed arc of the unit circle. More precisely, let $\gamma\in(0,\frac{1}{2})$ and let $$I=\{e^{2i\pi\theta};\;\; |\theta|\leq\gamma\}.$$
Then define $X_N$ to be the number of eigenvalues of $T_N$ in $I$. Using asymptotics of Toplitz determinants with discontinuous symbols, one can show that as $N\to\infty$, for all $|t|<\pi$,
$$\mathbb E\left[e^{it(X_N-2\gamma N)}\right]\sim \exp\left(-\frac{t^2}{2\pi^2}\log N\right)\left(2-2\cos 4\pi\gamma\right)^{\frac{t^2}{4\pi^2}} G\left(1-\frac{t}{2\pi}\right)G\left(1+\frac{t}{2\pi}\right),$$where $G$ is the Barnes double Gamma function. The restriction on $t$ is necessary since the characteristic function of $X_N$ is $2\pi$-periodic.

We would like now to report on two interesting examples of $a$-mod-Cauchy convergence related to arithmetics and which are in fact 
re-interpretations of results of Vardi~\cite{vardi} and of
Sarnak~\cite{sarnak}.
\par
First, recall that a Cauchy variable with parameter $\gamma>0$ is one
with law given by
$$
d\mu_{\gamma}=\frac{\gamma}{\pi}\frac{1}{\gamma^2+x^2} dx,
$$
and with characteristic function
$$
\int_{\Rr}{e^{itx}d\mu_{\gamma}(x)}=e^{-\gamma |t|},\quad\quad t\in\Rr.
$$
\par
The most natural definition of mod-Cauchy convergence would then be
that $(X_N)$ converges in mod-Cauchy sense with parameters
$(\gamma_N)$ and limiting function $\Phi$ if we have
$$
\lim_{N\ra +\infty}{\exp(\gamma_N|t|)\mathbb{E}[e^{itX_N}]}=\Phi(t)
$$
and the limit is locally uniform in $t$ (so $\Phi$ is continuous and
$\Phi(0)=1$). Let's say that we have $a$-mod-Cauchy
convergence if the limits above exist, locally uniformly, for
$|t|<a$ for some $a>0$. This restricted convergence is
sufficient to ensure the following:

\begin{fact}
  If $(X_N)$ converges $a$-mod-Cauchy sense with parameters
  $(\gamma_N)$ and some $a>0$, then we have convergence in law
$$
\frac{X_N}{\gamma_N}\Longrightarrow \mu_1.
$$
\end{fact}
We now detail our two examples of $a$-mod-Cauchy convergence from number theory. Note that, although they seem to involve very
different objects, they are in fact closely related through the way
they are proved using spectral theory for certain differential
operator involving complex multiplier systems on the modular surface
$SL(2,\mathbb Z)\backslash \mathbb{H}$.

\begin{example}[Dedekind sums](See~\cite{vardi}) The \emph{Dedekind
    sum} $s(d,c)$ is defined by
$$
s(d,c)=\sum_{h=1}^{d-1}{\Bigl(\Bigl(\frac{hd}{c}\Bigr)\Bigr)
\Bigl(\Bigl(\frac{h}{c}\Bigr)\Bigr)}
,\quad\quad ((x))=\begin{cases}
0&\text{ if $x$ is an integer}\\
x-\lfloor x\rfloor -1/2,&\text{ otherwise.}
\end{cases}
$$
for $1\leq d<c$ integers with $(c,d)=1$.
\par
Vardi proved the existence of a renormalized Cauchy limit for
$s(d,c)$: precisely, let 
$$
F_N=\{(c,d)\,\mid\, 1\leq d<c<N,\quad (c,d)=1\},
$$
for $N\geq 1$, and give it the probability counting measure $\mathbb{P}_N$
and expectation denoted $\mathbb{E}_N(\cdot)$. For any $a<b$, we then
have (\cite[Theorem 1]{vardi}) the limit
$$
\lim_{N\ra +\infty}{\mathbb{P}_N\Bigl(a<\frac{s(d,c)}{(\log
    c)/(2\pi)}<b\Bigr)}= \mu_{1}([a,b]).
$$
\par
or equivalently
(cf.~\cite[Prop. 1]{vardi})
$$
\lim_{N\ra +\infty}{\mathbb{P}_N\Bigl(a<\frac{s(d,c)}{(\log
    N)/(2\pi)}<b\Bigr)}= \mu_{1}([a,b]).
$$
The latter is obtained as consequence (using the Fact above)
of a restricted mod-Cauchy convergence.

\begin{theorem}[Vardi]
  Let $D_N$ be the random variable defined on $F_N$ by $(d,c)\mapsto
  s(d,c)$. Then, for any $\eps>0$, we have
$$
\mathbb{E}_N(e^{itD_N})=\exp(-\gamma_N|t|)\Phi(t)+O(N^{-2/3+\eps})
$$
uniformly for $|t|<2\pi$ where $\gamma_N=\frac{1}{2\pi}(\log N/4)$ and
$$
\Phi(t)=(1-\tfrac{|t|}{4\pi})^{-1}%\Bigl(\frac{1}{4}\Bigr)^{\tfrac{|t|}{2\pi}}
\ 
\Bigl(
\frac{3}{\pi}\int_{SL(2,\Zz)\backslash \mathbf{H}}{
(y|\eta(z)|^4)^{\tfrac{t}{2\pi}}\,\frac{dxdy}{y^2}
}
\Bigr)^{-1}
$$
the function $\eta(z)$ being the Dedekind eta function
$$
\eta(z)=e^{i\pi z/12}\prod_{n\geq 1}{(1-e^{2i\pi nz})}
$$
defined for $\Im(z)>0$.
\end{theorem}

\begin{proof}
  This follows from~\cite[Prop. 2]{vardi}, after making minor
  notational adjustments. In particular: Vardi uses $2\pi r$ instead
  of $t$; the case $t=0$ is omitted in Vardi's statement, but it is
  trivial; only the case $0<r<1$ is mentioned, but there is a symmetry
  $r\leftrightarrow -r$ (see~\cite[p. 7]{vardi}) that extends the
  result to $-1<r\leq 0$.
\end{proof}

Because
$$
\exp(-|t|\gamma_N)=\Bigl(\frac{N}{4}\Bigr)^{-\tfrac{t}{2\pi}},
$$
we see from the error term that the formula gives, in fact, only
restricted convergence with a well-defined limit for
$|t|<\frac{4\pi}{3}$.  It is not clear on theoretical grounds whether
this is optimal or not (note also the pole of the first factor of
$\Phi(t)$ for $t=\pm 4\pi$), but the numerical experiments summarized
in Figures 1 to 4, which illustrate the behavior of 
$$
\expect_N(e^{itD_N})\exp(\gamma_N|t|)
$$
for $N\leq 5000$ and $t\in \{\pi/2, \pi, 2\pi, 4\pi\}$, tend to
indicate that there is no limit when $t$ is large (note in particular
the $y$-scale for the last picture).
\begin{figure}[ht]
\centering
\includegraphics[width=5in]{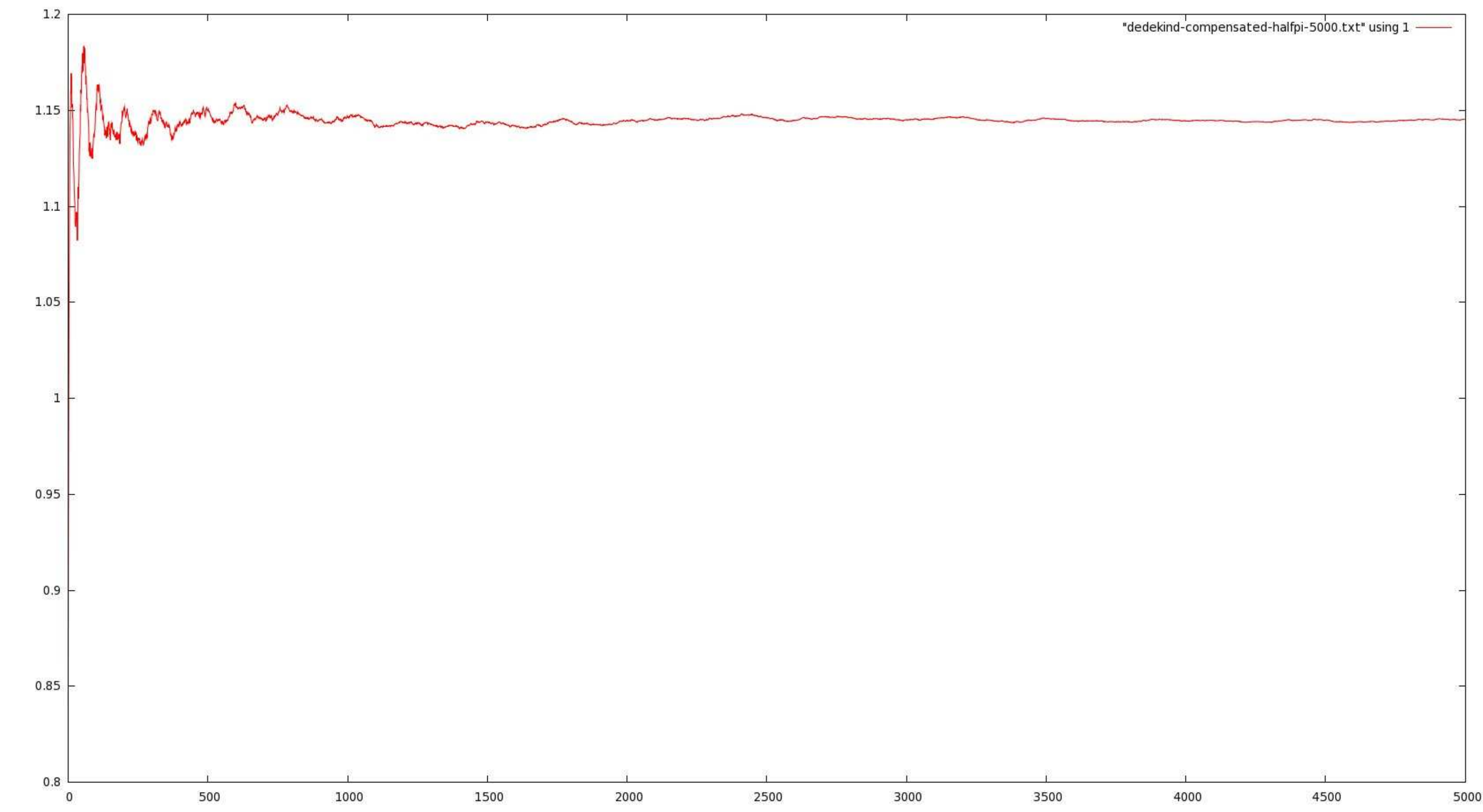}
\caption{$t=\pi/2$}
\end{figure}
\begin{figure}[ht]
\centering
\includegraphics[width=5in]{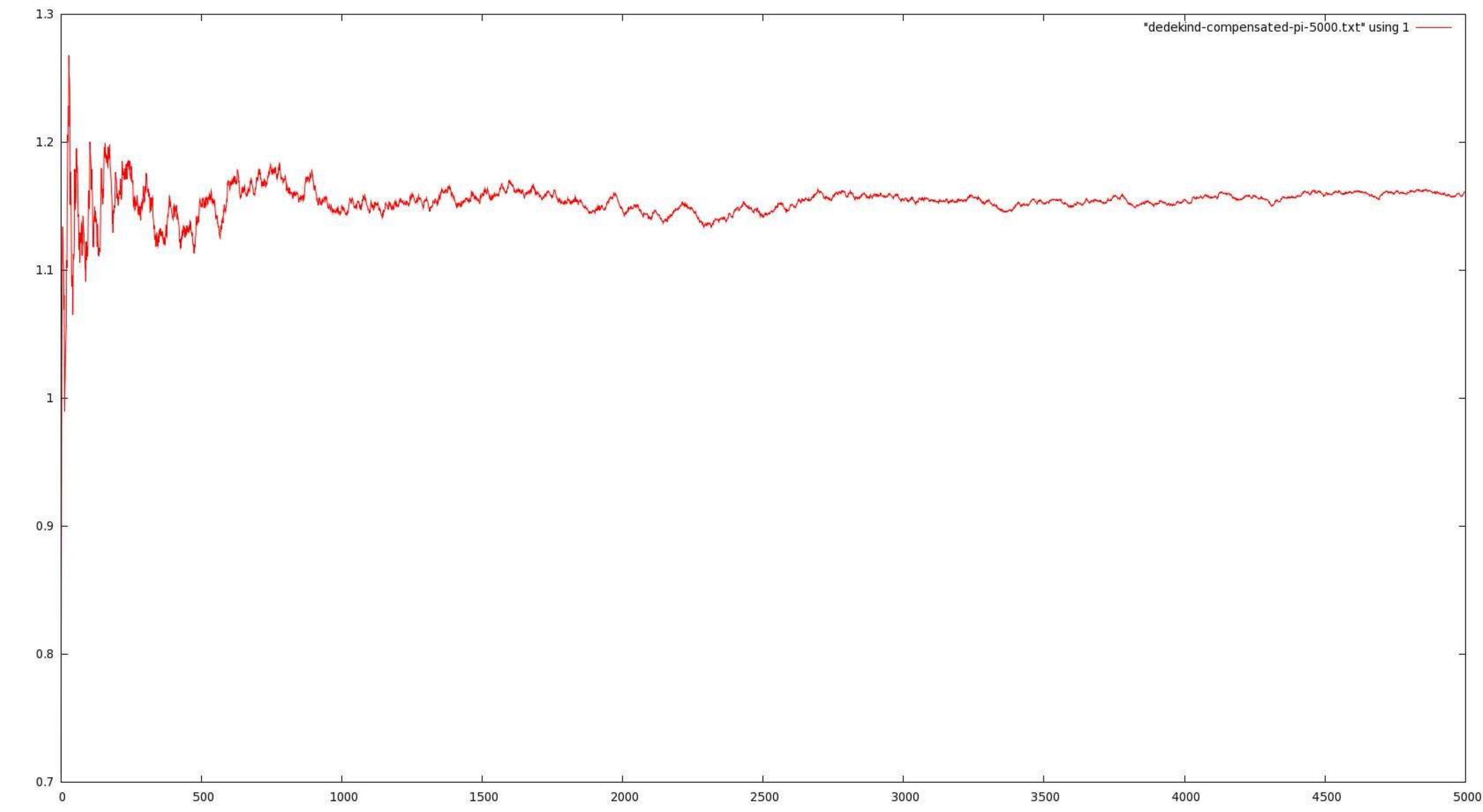}
\caption{$t=\pi$}
\end{figure}
\begin{figure}[ht]
\centering
\includegraphics[width=5in]{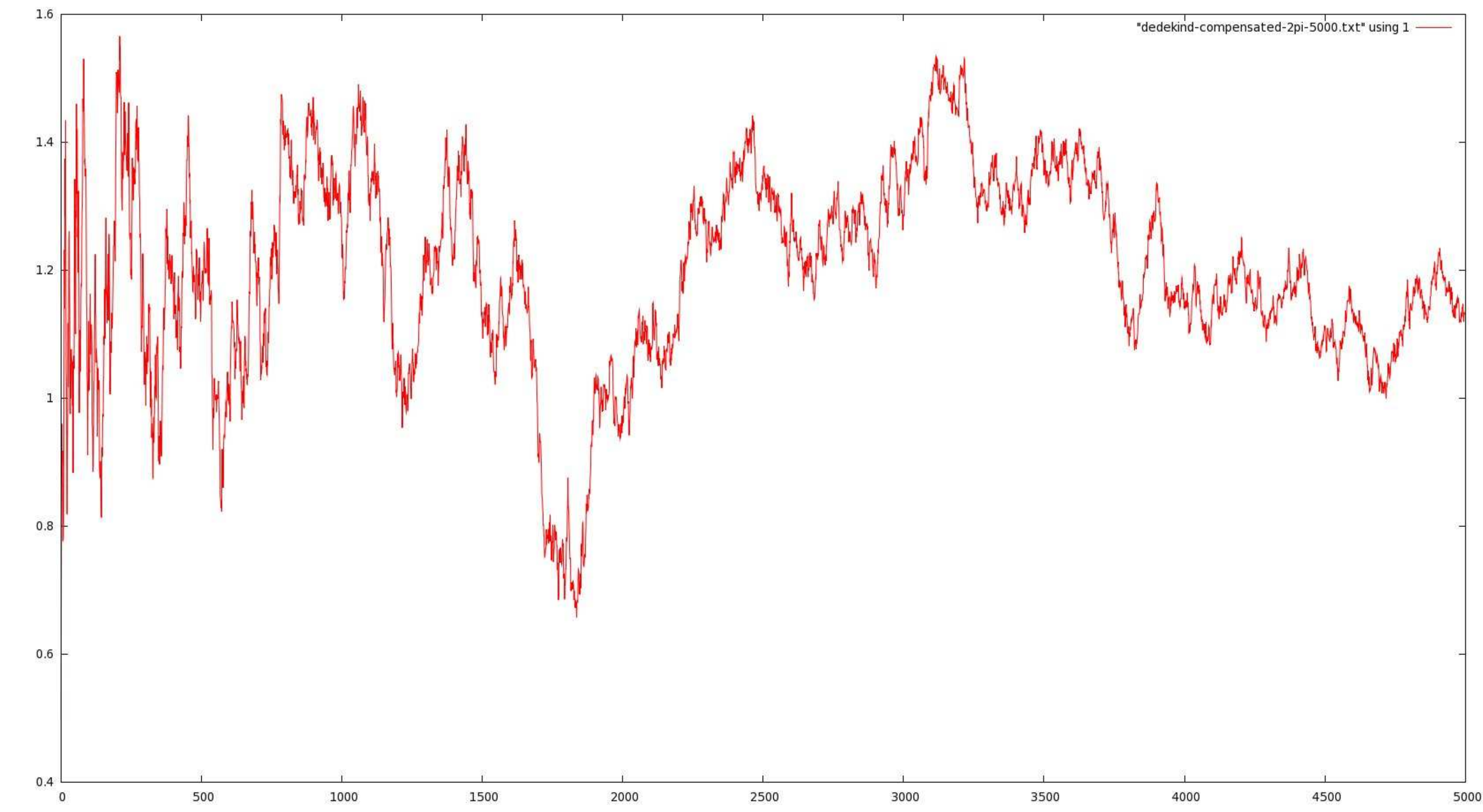}
\caption{$t=2\pi$}
\end{figure}
\begin{figure}[ht]
\centering
\includegraphics[width=5in]{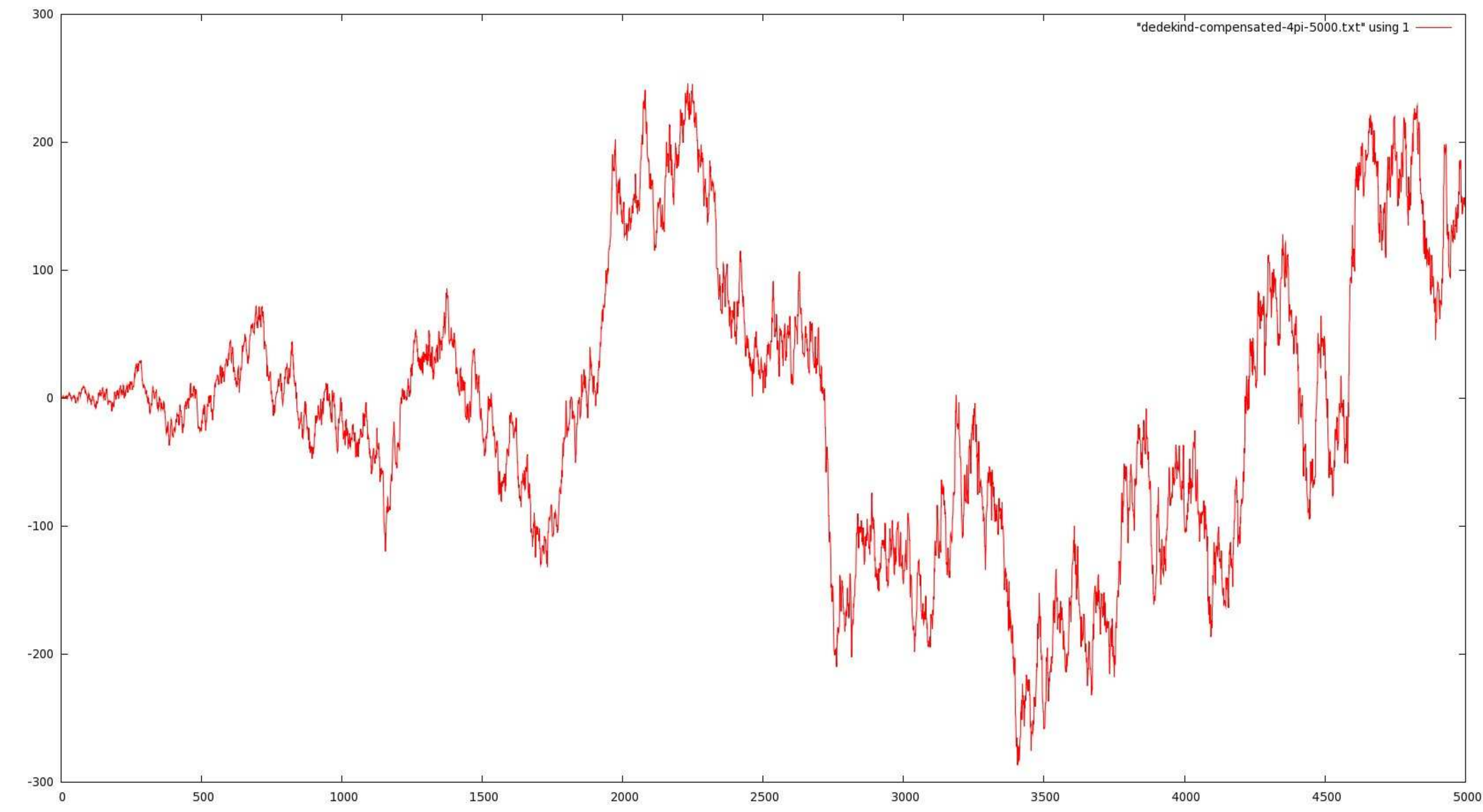}
\caption{$t=4\pi$}
\end{figure}
\par
Concerning the limiting function, recall that the measure
$$
\frac{3}{\pi}\frac{dxdy}{y^2}
$$
is a probability measure on the modular surface, so $\Phi(t)$
(surprisingly?) involves the inverse of a \emph{Laplace transform} of
the distribution function of $\log (y|\eta(z)|^4)$. 
\end{example}

\begin{example}[Linking numbers of modular geodesics]
  (See~\cite{sarnak} and~\cite{mozzochi}). The second example looks very different, as it
  concerns issues of geometry and topology. More precisely, following
  Ghys, Sarnak considers the asymptotic behavior of a map 
$$
C\mapsto \mathrm{lk}(k_C),
$$
where $C$ runs over the set $\Pi$ of prime closed geodedics in
$SL(2,\mathbb Z)\backslash \mathbb{H}$ and $\mathrm{lk}(k_C)$ is the linking
number of a knot associated to $C$ and the trefoil knot -- the
relation coming from an identification of the homogeneous space
$SL(2,\mathbb Z)\backslash SL(2,\mathbb R)$ with the complement in $\mathbb{S}^3$
of the trefoil knot. This is also accessible more concretely by the
classical identification of $\Pi$ with the set of primitive (i.e., not
of the form $g^n$, $n\geq 2$) hyperbolic (i.e., with $|\operatorname{Tr}(g)|>2$)
conjugacy classes in $SL(2,\mathbb Z)$. In this identification
$C\leftrightarrow g$, one has
$$
\mathrm{lk}(k_C)=\psi(g),
$$
where $\psi\,:\, PSL(2,\mathbb Z)\ra \mathbb Z$ is a fairly classical map (called
the Rademacher map), which is not a homomorphism but a
``quasi-homomorphism'' (namely, the map $(g,h)\mapsto
\psi(gh)-\phi(g)-\psi(h)$ is bounded on $PSL(2,\mathbb Z)^2$).  In turn, this $\psi$-function is related to
the multiplier system for the $\eta$ function. 
\par
Now, for $x>0$, let
$$
\Pi_x=\{g\in \Pi\,\mid\, N(g)\leq x\}
$$
where the ``norm'' $N(g)$ is defined and related to the length
$\ell(g)$ of the closed geodesic by
$$
N(g)=\Bigl(\frac{\operatorname{Tr}(g)+\sqrt{\operatorname{Tr}(g)^2-4}}{2}\Bigr)^2,\quad\quad
\ell(g)=\log N(g).
$$
\par
Let $\mathbb{P}_x$ denote the probability measure where each $g\in \Pi_x$
has weight proportional to $\ell(g)$; the normalizing factor to ensure
that it is a probability measure is
$$
\sum_{N(g)\leq x}{\log N(g)}\sim x
$$
as $x\ra +\infty$, by Selberg's Prime Geodesic Theorem (this can be
made much more precise, see e.g.~\cite{iwaniec-pgt}). Let $\mathbb{E}_x$ denote the corresponding
expectation operator.
\par
Sarnak~\cite[Theorem 3]{sarnak} (see also the detailed proofs by Mozzochi~\cite{mozzochi})proves a limiting Cauchy behavior:
$$
\lim_{x\ra
  +\infty}{\mathbb{P}_x\Bigl(a<\frac{\mathrm{lk}(g)}{\ell(g)}<b\Bigr)}
=\mu_1([a,b])
$$
for any $a<b$.
\par
Again, if one looks at the proof, one sees that this is deduced from:

\begin{theorem}[Sarnak]
  Let $\mathrm{lk}_x$ denote the random variable $g\mapsto
  \mathrm{lk}(g)=\psi(g)$ on $\Pi_x$. Then for $|t|\leq \pi/12$, we
  have
$$
\mathbb{E}_x(e^{it \mathrm{lk}_x})=\exp(-|t|\gamma_x)\Phi_1(t)+O(x^{3/4})
$$
where $\gamma_x=\frac{3}{\pi}(\log x)$ and
$\Phi_1(t)=\frac{1}{1-\frac{3|t|}{\pi}}$. 
\end{theorem}

\begin{proof}
  Again, up to notational changes, this is given
  by~\cite[(16)]{sarnak} since the quantity $v_r(\gamma)$ there is
  given by
$$
v_r(g)=e^{i\pi r \psi(g)/6}
$$
for $g\in \Pi$ and $r\in\mathbb R$. So the $r$ in loc. cit. is given by
$r=6t/\pi$ to recover our formulation. 
\end{proof}

This is again an example of restricted mod-Cauchy convergence. Again,
we do not know how far the restriction on $t$ is necessary.  One may of
course perform a summation by parts to remove the weight $\log
N(g)=\ell(g)$ from these results, if desired.

\end{example}

\end{document}